\documentclass[a4paper,12pt,reqno]{amsart}
\usepackage{amsmath, amssymb,amsthm}
\usepackage{enumerate}
\usepackage{cite}
\usepackage{graphicx}
\usepackage{subcaption}
\usepackage{kpfonts}
\usepackage{pifont}
\numberwithin{equation}{section}
\newtheorem{theorem}{Theorem}[section]

\newtheorem{proposition}[theorem]{Proposition}

\newtheorem{definition}[theorem]{Definition}
\newtheorem{thmx}{Theorem}

\theoremstyle{definition}

\begin{document}
\title[Quadratic transformation:
an $\mathcal{LU}$ factorization approach]
{Quadratic transformation and matrix biorthogonal polynomials:
an $\mathcal{LU}$ factorization approach}


\author[K. K. Behera]{Kiran Kumar Behera}
\address{
Department of Mathematics,
Indian Institute of Science Bangalore-560012,
Karnataka, India}
\curraddr{}
\email{kiranbehera@iisc.ac.in}
\thanks
{This research is supported by the Dr. D. S. Kothari postdoctoral
fellowship scheme of University Grants Commission (UGC), India.
}




\date{}

\dedicatory{}

\commby{}

\begin{abstract}
The manuscript presents the $LU$ approach to matrix
biorthogonal polynomials when all the even ordered entries
in the Gram matrix are zero.
This arises in case of a quadratic transformation which is briefly
discussed.
Further, the main diagonal
of the Gram matrix is a zero diagonal and we present the theory
that follows from this fact.
Precisely, we discuss the Christoffel transformation and matrix representations
of the kernel polynomials, usually called the ABC Theorem.
Finally, we provide an illustration of our results assuming the Gram matrix
has Hankel symmetry.
\end{abstract}

\maketitle

\section{Introduction}
\label{sec: Introduction}
In recent years, the Gaussian (or $\mathcal{L}\mathcal{U}$)
decomposition technique
has been used as the basis for an alternative approach
for instance, to the theories of generalized orthogonal polynomials,
multiple orthogonal polynomials, integrable systems and
multivariate orthogonal polynomials
\cite{Adler-group-factorization-1997,
Adler-momet-matrices-2009,
Manas-Ulysess-multiple-Gauss-Borel-2011,
Manas-multivariate-AdvM-2016}.
We specially refer to
\cite{Manas-CD-LU-Survey-2021} and references therein
for a systematic theory of $\mathcal{L}\mathcal{U}$
factorization in case of matrix biorthogonal polynomials.
This approach brings the Gram matrix with
respect to a bilinear functional
to the center of analysis and usually the
starting point is the $\mathcal{L}\mathcal{U}$
decomposition of the Gram matrix.

A bilinear form on the ring $\Pi_n[z]$ of polynomials in the variable $z$,
where $\Pi_n$ is the ring of $n\times n$ matrices, is the map
\begin{align*}
\langle\cdot,\cdot\rangle: \Pi_n[z]\times\Pi_n[z]\mapsto\Pi_n,
\end{align*}
satisfying the following properties
\begin{enumerate}[(i)]
\item
$\langle\mathcal{C}_1\mathcal{P}(z)+\mathcal{C}_2\mathcal{Q}(z),
\mathcal{R}(\omega)\rangle
=
\mathcal{C}_1\langle\mathcal{P}(z),\mathcal{R}(\omega)\rangle
+
\mathcal{C}_2\langle\mathcal{Q}(z),\mathcal{R}(\omega)\rangle$,
\item
$\langle\mathcal{P}(z),
\mathcal{C}_1\mathcal{Q}(\omega)+\mathcal{C}_2\mathcal{R}(\omega)\rangle
=
\langle\mathcal{P}(z),\mathcal{Q}(\omega)\rangle\mathcal{C}_1^{T}
+
\langle\mathcal{P}(z),\mathcal{R}(\omega)\rangle\mathcal{C}_2^{T}$,
\end{enumerate}
for any $\mathcal{C}_1$, $\mathcal{C}_2$ $\in\Pi_n$
and
$\mathcal{P}(z), \mathcal{Q}(z), \mathcal{R}(z)\in\Pi_n[z]$.
Further, if
\begin{align*}
\mathcal{P}(z)=\sum_{k=0}^{m}p_kz^k,
\quad
\mathcal{Q}(z)=\sum_{l=0}^{n}q_lz^l,
\quad
p_k,q_k\in\Pi_n,
\end{align*}
the bilinear form acts as
\begin{align*}
\langle\mathcal{P}(z), \mathcal{Q}(\omega)\rangle=
\sum_{k=0}^{m}\sum_{l=0}^{n}p_km_{kl}q_l,
\quad
m_{kl}=\langle z^{k}\mathcal{I}, \omega^{l}\mathcal{I}\rangle,
\end{align*}
where $\mathcal{I}$ is the identity matrix in $\Pi_n$.
This gives rise to the Gram matrix $\mathcal{M}$ where
\begin{align}
\label{eqn: gram matrix definition}
\mathcal{M}=[m_{i,j}]
=
\left(
  \begin{array}{ccc}
    m_{00} & m_{01}  & \cdots \\
    m_{10} & m_{11}  & \cdots \\
    \vdots & \vdots   & \ddots \\
  \end{array}
\right),
\quad
i,j\geq0.
\end{align}
We also note that if
$\mathcal{X}(z):=
(\mathcal{I}, z\mathcal{I}, z^2\mathcal{I}, \cdots)^{T}$,
then $\mathcal{M}$ has the alternative description
%
$\mathcal{M}=
\langle\mathcal{X}(z),\mathcal{X}(\omega)\rangle$.

The rudiments of this approach are as follows. The bilinear form $\langle\cdot, \cdot\rangle$
is said to be quasi-definite whenever the associated Gram matrix has all its leading principal
minors different from zero. In this case, the Gram matrix has a unique
$\mathcal{L}\mathcal{U}$  decomposition,
but in fact, is written in the form
$\mathcal{M}=\mathcal{L}_1^{-1}\mathcal{D}\mathcal{L}_2^{-T}$,
where $\mathcal{D}$ is a diagonal matrix,
and
$\mathcal{L}_1$, $\mathcal{L}_2$ are lower triangular matrices
with $\mathcal{I}$ as the diagonal entries.

Two vectors
$\mathcal{P}(z):=\mathcal{L}_1\mathcal{X}(z)$
and
$\mathcal{Q}(\omega):=\mathcal{L}_2\mathcal{X}(\omega)$
are defined so that if
\begin{align}
\label{eqn: definition of polynomial P and Q vectors}
\mathcal{P}(z)=\left(
                 \begin{array}{ccc}
                   p_0(z) & p_1(z) & \cdots \\
                 \end{array}
               \right)^T,
\quad
\mathcal{Q}(\omega)=\left(
                 \begin{array}{ccc}
                   q_0(\omega) & q_1(\omega) & \cdots \\
                 \end{array}
               \right)^T,
\end{align}
then $p_k(z)$ and $q_l(\omega)$ are
matrix polynomials whose coefficients are easily
determined from $\mathcal{L}_1$ and $\mathcal{L}_2$ respectively.
This serves as the foundation to discuss various properties related to
biorthogonality, spectral transformations like the Christoffel
and Geronimus transformations and so on.
We note down three such properties, or rather representations
for future reference.

Suppose the underlying bilinear form is quasi-definite.
Then we have the following representation
\cite{Manas-CD-LU-Survey-2021}
\begin{align}
\label{eqn: quasi-determinantal-pn(z)-introduction}
p_n(z)=\theta_{\ast}\left(
         \begin{array}{ccc|c}
            &  &  & 1 \\
            & \mathcal{M}^{[n]} &  & \vdots \\
            &  &  & z^{n-1} \\\hline
           m_{n,0} & \cdots & m_{n,n-1} & z^n \\
         \end{array}
       \right),
\quad
n=0,1,\cdots,
\end{align}
where the superscript $[n]$ denotes a truncated matrix
and
$\theta_{\ast}$ denotes the fact that $p_n(z)$
is, in fact, the Schur complement of $\mathcal{M}^{[n]}$
in the above matrix. A similar expression also exists for $q_n(\omega)$
and we will explain this in greater detail in Section~\ref{sec: Generic LDU}.
Further, given the biorthogonal matrix polynomial sequences
$\{p_n(z)\}$ and $\{q_n(\omega)\}$,
the $j^{th}$ Christoffel-Darboux kernel polynomial is defined as
\cite{Manas-CD-LU-Survey-2021}
\begin{align}
\label{eqn: definition kernel polynomial-introduction}
\mathcal{K}^{[j]}(z,\omega)=\sum_{k=0}^{j}q_k^{T}(\omega)d_{kk}^{-1}p_{k}(z),
\quad
\mathcal{D}=\mbox{diag}~(d_{00},\,d_{11},\,\cdots).
\end{align}
These matrix kernel polynomials also possess the following representation,
the so called ABC Theorem (named after Aitken, Berg and Collar)
\cite{Simon-CD-kernel-Survey}
\begin{align}
\label{eqn: definition kernel polynomial-ABC theorem-introduction}
\mathcal{K}^{[j]}(z,\omega)=(\mathcal{X}^{[j]}(\omega))^{T}[\mathcal{M}^{[j]}]^{-1}
\mathcal{X}^{[j]}(z),
\end{align}
which is a direct consequence of
\eqref{eqn: definition kernel polynomial-introduction}.

Henceforward, for ease of reference we will
call $m_{ij}$ the even (odd) ordered Gram entry
if $i+j$ is even (odd). Our primary concern
in the present manuscript is to discuss
the above factorization technique when all the
even ordered Gram entries are zero matrices.
Note that in such a case, the main diagonal of $\mathcal{M}$
is the zero diagonal and hence the underlying bilinear form is no
longer quasi-definite.
This forces the main diagonal of the matrix $\mathcal{D}$
in the ensuing $\mathcal{L}\mathcal{D}\mathcal{U}$ decomposition
to consist of zero entries and hence it has to be substituted appropriately.
This will lead to a revision of the definition
\eqref{eqn: definition kernel polynomial-introduction} of the
matrix kernel polynomials and consequently the ABC Theorem
\eqref{eqn: definition kernel polynomial-ABC theorem-introduction}.
Among other things, the quasi-determinant
representations \eqref{eqn: quasi-determinantal-pn(z)-introduction}
will also be no longer valid since for $p_0(z)$, $\mathcal{M}^{[0]}$ is the
zero matrix.
The precise goal of the present manuscript
is to address the above concerns.

The layout of the manuscript is as follows.
In rest of the present section, we briefly discuss the theory where
Gram matrices with even ordered Gram entries equal to zero arise.
In Section~\ref{sec: Generic LDU}, we obtain the $LDU$
factorization of such Gram matrices,
biorthogonality relations and representations of
matrix biorthogonal polynomials.
Section~\ref{sec: Christoffel transformation} demonstrates
the Christoffel transformation of such Gram matrices.
The associated matrix kernel polynomials followed by
the appropriate form of the ABC Theorem are presented.
Section~\ref{sec: Hankel symmetry} provides an illustration
of the results in the special case of Hankel symmetry.
\subsection{Unwrapping of measure}
The quadratic transformation $\lambda\mapsto\lambda^2$ is perhaps the
simplest case in the general theory of polynomial mappings.
One of the direction in which this has been studied is the following
\cite{Marcellan-OP-quadratic-1999}.
Given $\{p_n(\lambda)\}$, a monic orthogonal polynomial sequence (MOPS),
the problem is to find another MOPS
$\{s_n(\lambda)\}$ such that
$s_{2k}(\lambda)=p_k(t(\lambda))$, $k\geq0$, where $t(\lambda)$ is a monic polynomial
of degree 2.
The sequence $\{s_n(\lambda)\}$ is then completed by defining
$s_{2k+1}(\lambda)=(\lambda-a)p_{k}(t(\lambda))$.
For some other approaches, we refer to
\cite{Bessis-Moussa-iterated-mapping-CommMathPhy-1983,
Geronimo-Van-Assche-several-interval-TAMS-1988,
Ismail-sieved-III-TAMS-1986,
Ismail-Charris-Sieved-VII-TAMS-1993}.

The connection of the transformation $\lambda\mapsto\lambda^2$
with moment theory in the context of present manuscript arises in what is known
as unwrapping of measures.
To begin with, given a sequence of $p\times p$ Hermitian matrices,
$S_0, S_1, \cdots, S_{2n}$, the
truncated Hamburger (TH) moment problem
 is to find a Hermitian matrix measure
$\mu(u)$, $-\infty<u<\infty$, such that
\begin{align}
\label{eqn: matrix Hamburger problem}
S_k=\int_{-\infty}^{\infty}u^kd\mu(u),
\quad
k=0,1,\cdots,2n.
\end{align}
The Hamburger-Nevanlinna Theorem solves
\eqref{eqn: matrix Hamburger problem}
as an interpolation problem in the class $\mathcal{N}_p$
of Nevanlinna functions and is presented below.
\begin{thmx}[\cite{Akhiezer-book,Chen-matrixCF-LAA-1998,Kovalishina-1983}]
\label{theorem: Theorem A-introduction}
If $\mu(u), (-\infty<u<\infty)$ is a solution to the TH problem
\eqref{eqn: matrix Hamburger problem},
then there exists $\mathcal{F}(\lambda)\in\mathcal{N}_{p}$
such that
\begin{align}
\label{eqn: integral representation-introduction}
\mathcal{F}(\lambda)=
\int_{-\infty}^{\infty}\frac{1}{u-\lambda}d\mu(\lambda)
\end{align}
for which
\begin{align}
\label{eqn: asymptotic series-trucated-introduction}
\lim_{\lambda\rightarrow\infty}
\lambda^{2n+1}\left[
\mathcal{F}(\lambda)+\frac{S_0}{\lambda}+\frac{S_1}{\lambda^2}+\cdots+\frac{S_{2n-1}}{\lambda^{2n}}
\right]
=
-S_{2n}
\end{align}
uniformly in the sector
$\pi_{\epsilon}:=\epsilon<\arg{\lambda}<\pi-\epsilon$,
for some $\epsilon\in(0,\pi/2)$.
Conversely, if \eqref{eqn: asymptotic series-trucated-introduction}
holds, at least for $\lambda=iy$
$(y\rightarrow+\infty)$, for some $\mathcal{F}(\lambda)\in\mathcal{N}_p$,
then $\mathcal{F}(\lambda)$ has the representation
\eqref{eqn: integral representation-introduction},
where $\mu(u)$ has $2n+1$ moments $S_0, S_1, \cdots, S_{2n}$.
\end{thmx}
If the relation \eqref{eqn: asymptotic series-trucated-introduction}
holds for all $n=0,1,\cdots$, we have the asymptotic expansion
\begin{align}
\label{eqn: asymptotic series-for-all-n-introduction}
\mathcal{F}(\lambda)\sim
-\frac{S_0}{\lambda}-\frac{S_1}{\lambda^2}-\cdots-\frac{S_{2n-1}}{\lambda^{2n}}-
\cdots,
\quad\lambda\in\pi_{\epsilon}.
\end{align}
Hence, if we let $\mathcal{H}(\lambda)=\mathcal{F}(\lambda^2)$,
then from \eqref{eqn: asymptotic series-for-all-n-introduction}, we have
\begin{align*}
\mathcal{H}(\lambda)\sim
-\frac{h_0}{\lambda}-\frac{h_1}{\lambda^2}-\cdots-\frac{h_{2n-1}}{\lambda^{2n}}-
\cdots,
\quad\lambda\in\pi_{\epsilon},
\end{align*}
where $h_{2k=0}$, $h_{2k+1}=S_k$, $k\geq0$, giving rise to the following
moment matrix
\begin{align}
\label{eqn: definition of Gram matrix with even zero}
\mathcal{M}
=
\left(
  \begin{array}{ccccc}
    0 & m_{0,1} & 0 & m_{0,3} & \cdots \\
    m_{1,0} & 0 & m_{1,2} & 0 & \cdots \\
    0 & m_{2,1} & 0 & m_{2,3} & \cdots \\
    m_{3,0} & 0 & m_{3,2} & 0 & \cdots \\
    \vdots & \vdots & \vdots & \vdots & \ddots \\
  \end{array}
\right),
\quad
m_{i,j}=h_{i+j},
\quad
i,j=0,1,\cdots.
\end{align}
The motivation to introduce quadratic transformation in
\eqref{eqn: asymptotic series-for-all-n-introduction}
comes from the fact that
if we consider the problem \eqref{eqn: matrix Hamburger problem}
over the interval $[0,\infty)$,
we have the truncated Stieltjes (TS) moment problem
which is solved by $\varphi(\lambda)$ in the Stieltjes class
$\mathcal{S}_p$.
A result like Theorem~\ref{theorem: Theorem A-introduction} exists
\cite{Narcowich-R-operators-IndianaJ-1977}
for the TS problem in which
$\varphi(\lambda)$ satisfies
\eqref{eqn: asymptotic series-trucated-introduction}
and
\eqref{eqn: matrix Hamburger problem}, but over the interval $[0,\infty)$.
Further, it is well known that
\cite{Maxim-polynomial-mapping-JAT-2013,
Simon-moment-self-adjoint-Advances-1998,Wall-book}
\begin{align*}
\varphi(\lambda)=\int_{0}^{\infty}\frac{1}{u-\lambda}d\mu(u)
\Longrightarrow
\lambda\varphi(\lambda^2)=\int_{-\infty}^{\infty}\frac{\mbox{sign} u}{2}\frac{1}{u-\lambda}d\mu(u^2),
\end{align*}
thus reducing a Stieltjes moment problem to a Hamburger one.
The quadratic transformation $\lambda\mapsto\lambda^2$
is the first part of this phenomenon of unwrapping of measures and provides the
background for the problem under consideration.

\section{The generic $LDU$ decomposition}
\label{sec: Generic LDU}
In this section we find the
$\mathcal{L}\mathcal{D}\mathcal{U}$
decomposition of the Gram matrix $\mathcal{M}$
(without assuming Hankel symmetry)
given by
\eqref{eqn: definition of Gram matrix with even zero},
where $\mathcal{L}$, $\mathcal{D}$, $\mathcal{U}$
have the forms
\begin{align*}
\mathcal{L}
=
\left(
  \begin{array}{cccc}
    \mathcal{I}_2 & 0 & 0 & \cdots \\
    \mathcal{L}_{10} & \mathcal{I}_2 & 0 & \cdots \\
    \mathcal{L}_{20} &\mathcal{L}_{21} & \mathcal{I}_{2} & \cdots \\
    \vdots & \vdots & \vdots & \ddots \\
  \end{array}
\right),
\quad
\mathcal{U}
=
\left(
  \begin{array}{cccc}
    \mathcal{I}_2 & \mathcal{U}_{01} & \mathcal{U}_{02} & \cdots \\
    0 & \mathcal{I}_2 & \mathcal{U}_{12} & \cdots \\
    0 & 0 & \mathcal{I}_2 & \cdots \\
    \vdots & \vdots & \vdots & \ddots \\
  \end{array}
\right)
\end{align*}
and
\begin{align*}
\mathcal{D}
=
\left(
  \begin{array}{cccc}
    \mathcal{D}_{00} & 0 & 0 & \cdots \\
    0 & \mathcal{D}_{11} & 0 & \cdots \\
    0 & 0 & \mathcal{D}_{22} & \cdots \\
    \vdots & \vdots & \vdots & \ddots\\
  \end{array}
\right),
\end{align*}
with $l_{i,j}, u_{i,j}, d_{i,i}\in\Pi_n$ and
\begin{align}
\label{eqn: form of constituent matrices}
\begin{split}
\mathcal{L}_{ij}
&=
\left(
  \begin{array}{cc}
    l_{2i,2j} & 0 \\
    0 & l_{2i+1,2j+1} \\
  \end{array}
\right),
\quad
\mathcal{U}_{ij}
=
\left(
  \begin{array}{cc}
    u_{2i,2j} & 0 \\
    0 & u_{2i+1,2j+1} \\
  \end{array}
\right),
\\
\mathcal{D}_{ii}
&=
\left(
  \begin{array}{cc}
    0 & d_{2i,2i+1} \\
    d_{2i+1,2i} & 0 \\
  \end{array}
\right),
\quad
\mathcal{I}_{2}
=
\left(
  \begin{array}{cc}
    \mathcal{I} & 0 \\
    0 & \mathcal{I}\\
  \end{array}
\right).
\end{split}
\end{align}
Observe that $\mathcal{L}$ and $\mathcal{U}$ are block triangular
while $\mathcal{D}$ is block diagonal. Hence one of the ways to proceed
is to opt for the block $\mathcal{L}\mathcal{D}\mathcal{U}$
factorization.
%
But, because of the special structure of the constituent blocks,
we actually find the expressions for $l_{ij}$, $u_{ij}$ and $d_{ij}$.

We do this in two steps. First, we let
$\mathcal{H}=(h_{ij})=\mathcal{L}\mathcal{D}$
so that from $\mathcal{M}=\mathcal{H}\mathcal{U}$,
we have using the fact that $u_{l,l}=\mathcal{I}$
\begin{align}
\label{eqn: relations for m-2j and m-2j+1}
\begin{split}
m_{i,2k}
&=
h_{i,0}u_{0,2k}+h_{i,2}u_{2,2k}+\cdots+h_{i,2k},
\quad
i=1,3\cdots,
\\
m_{i,2k+1}
&=
h_{i,1}u_{1,2k+1}+h_{i,3}u_{3,2k+1}+\cdots+h_{i,2k+1},
\quad
i=0,2\cdots,
\end{split}
\end{align}
for $k=0,1,2,\cdots,$. Using an algorithmic approach, we determine $h_{ij}$ and $u_{ij}$
from \eqref{eqn: relations for m-2j and m-2j+1}. In the second
step, we determine $l_{ij}$ and $d_{ij}$ from $h_{ij}$.

To proceed, let us work out the special cases $m_{i,2}$ and $m_{i,4}$.
We will repeatedly use the facts that
$h_{i,j}=0$ if $j=i+1$ for $i=1,3,5,\cdots$ or $i+j$ is even or $j\geq i+2$.
Since $m_{i,2}=h_{i,0}u_{0,2}+h_{i,2}$, we have
$m_{1,2}=h_{1,0}u_{0,2}$ for $i=1$ which implies
$u_{0,2}=h_{1,0}^{-1}m_{1,2}$.
Then,
\begin{align*}
h_{i,2}=m_{i,2}-h_{i,0}h_{1,0}^{-1}m_{1,2}=
\theta_{\ast}\left(
               \begin{array}{cc}
                 h_{1,0} & m_{1,2} \\
                 h_{i,0} & m_{i,2} \\
               \end{array}
             \right)
=
\theta_{i,2}^{(2)},
\quad\mbox{say}.
\end{align*}
Next, from $m_{i,4}=h_{i,0}u_{0,4}+h_{i,2}u_{2,4}+h_{i,4}$,
we substitute $i=1$ to obtain (since $h_{1,2}=h_{1,4}=0$)
$u_{0,4}=h_{1,0}^{-1}m_{1,4}$.
Observe that $h_{i,2}$ has been obtained in the previous step for $m_{i,2}$.
So we put $i=3$ to obtain
\begin{align*}
h_{3,2}u_{2,4}=
m_{3,4}-h_{3,0}h_{1,0}^{-1}m_{1,4}=
\theta_{\ast}
\left(
  \begin{array}{cc}
    h_{1,0} & m_{1,4} \\
    h_{3,0} & m_{3,4} \\
  \end{array}
\right)
=
\theta_{3,4}^{(2)}, \quad\mbox{say}.
\end{align*}
Then, $u_{2,4}=h_{3,2}^{-1}\theta_{3,4}^{(2)}$. Finally,
\begin{align*}
h_{i,4}
=
m_{i,4}-h_{i,0}u_{0,4}-h_{i,2}u_{2,4}
=
\theta_{\ast}
\left(
  \begin{array}{cc}
    h_{1,0} & m_{1,4} \\
    h_{i,0} & m_{i,4} \\
  \end{array}
\right)
-
h_{i,2}h_{3,2}^{-1}\theta_{3,4}^{(2)}.
\end{align*}
Then, proceeding as in the case of $m_{1,2}$ we have
\begin{align*}
h_{i,4}=
\theta_{\ast}
\left(
  \begin{array}{cc}
    h_{3,2} & \theta_{3,4}^{(2)} \\
    h_{i,2} & \theta_{i,4}^{(2)} \\
  \end{array}
\right)
=
\theta_{i,4}^{(4)},
\quad\hbox{say, where}
\theta_{i,4}^{(2)}:=
\theta_{\ast}
\left(
  \begin{array}{cc}
    h_{1,0} & m_{1,4} \\
    h_{i,0} & m_{i,4} \\
  \end{array}
\right).
\end{align*}
Now, it is matter of induction to prove the following.
\begin{theorem}
Define for $i=1,3,5,\cdots$
\begin{align*}
\theta_{i,2k}^{(2j)}=
\theta_{\ast}
\left(
  \begin{array}{cc}
    h_{2j-1,2j-2} & \theta_{2j-1,2k}^{(2j-2)} \\
    h_{i,2j-2} & \theta_{i,2k}^{(2j-2)} \\
  \end{array}
\right),
\quad
j=1,2,\cdots,
\quad
k=0,1,\cdots,
\end{align*}
where $\theta_{i,2k}^{(0)}=m_{i,2k}$.
Then
\begin{align}
\label{eqn: induction hypothesis theorem}
h_{i,2l}
=\theta_{i,2l}^{(2l)},
\quad
u_{2l,2k}
=
h_{2l+1,2l}^{-1}\theta_{2l+1,2k}^{(2l)},
\quad
l=0,1,\cdots.
\end{align}
\end{theorem}
\begin{proof}
Observe that we have actually proved the theorem for the cases
$j=0,1$ and $k,l=0,1,2$
in the preceding discussion.
Hence, we assume that
\eqref{eqn: induction hypothesis theorem} holds for
$l=0,1,\cdots,j-1$ and $k=j$
and prove that the same holds for $l=j$.
For a simpler form, we note
\begin{align}
\label{eqn: simpler form for theta-induction}
\theta_{i,2k}^{(2j)}=
\theta_{i,2k}^{(2j-2)}-h_{i,2j-2}h_{2j-1,2j-2}^{-1}
\theta_{2j-1,2k}^{(2j-2)},
\quad
\theta_{i,2k}^{(0)}=m_{i,2k}.
\end{align}
We first obtain $u_{2j,2k}$
from the relation for $m_{i,2k}$
in \eqref{eqn: relations for m-2j and m-2j+1}
by substituting $i=2j+1$, $j<k$.
Recall that $h_{2j+1,2j+2}=0$ and $h_{2j+1,\ast}=0$ for $\ast-(2j+1)>1$.
Then,
\begin{align*}
h_{2j+1,2j}u_{2j,2k}
&=
m_{2j+1,2k}-h_{2j+1,0}u_{0,2k}-h_{2j+1,2}u_{2,2k}-
\cdots-h_{2j+1,2j-2}u_{2j-2,2k},
\\
&=
\theta_{2j+1,2k}^{(2)}-h_{2j+1,2}h_{3,2}^{-1}\theta_{3,2k}^{(2)}-\cdots-h_{2j+1,2j-2}u_{2j-2,2k},
\\
&=\cdots
\\
&=\theta_{2j+1,2k}^{(2j-2)}-h_{2j+1,2j-2}h_{2j-1,2j-2}^{-1}\theta_{2j-1,2k}^{(2j-2)}
=\theta_{2j+1,2k}^{(2j)}
\end{align*}
This gives
$u_{2j,2k}=h_{2j+1,2j}^{-1}\theta_{2j+1,2k}^{(2j)}$,
thereby proving the expression
for $u_{2l,2k}$ for $l=j$.
With $u_{2j,2k}$ determined, we put $k=j$ in the relation for $m_{i,2k}$
from \eqref{eqn: relations for m-2j and m-2j+1}
to obtain similarly
\begin{align*}
h_{i,2j}
&=
\theta_{i,2j}^{(0)}-h_{i,0}h_{1,0}^{-1}\theta_{1,2j}^{(0)}-h_{i,2}u_{2,2j}-
h_{i,4}u_{4,2j}-\cdots-h_{i,2j-2}u_{2j-2,2j},
\\
&=\cdots
\\
&=\theta_{i,2j}^{(2j-2)}-h_{i,2j-2}h_{2j-1,2j-2}^{-1}\theta_{2j-1,2j}^{(2j-2)}
=\theta_{i,2j}^{(2j)}.
\end{align*}
This proves the relation
\eqref{eqn: induction hypothesis theorem} for $l=j$, thus
completing the proof.
\end{proof}
\subsection{Biorthogonality relations}
Having decomposed the moment matrix
\eqref{eqn: definition of Gram matrix with even zero}
which henceforward, we write in the form
$\mathcal{M}=\mathcal{L}_1^{-1}\mathcal{T}\mathcal{L}_2^{-T}$,
let us define two block vectors with polynomial entries as
\begin{align}
\label{eqn: definition of P and Q with L1-L2-X(z)-X(omega)}
\mathcal{P}(z)=\mathcal{L}_1\mathcal{X}(z),
\quad
\mathcal{Q}(\omega)=\mathcal{L}_2\mathcal{X}(\omega),
\end{align}
where $\mathcal{P}(z)$ and $\mathcal{Q}(\omega)$
are as in \eqref{eqn: definition of polynomial P and Q vectors}.
Explicitly for $\mathcal{P}(z)$ we have
\begin{align*}
\left(
  \begin{array}{ccccc}
    \mathcal{I} & 0 & 0 & 0 & \cdots \\
    0 & \mathcal{I} & 0 & 0 & \cdots \\
    L_{20}^{(1)} & 0 & \mathcal{I} & 0 & \cdots \\
    0 & L_{31}^{(1)} & 0 & \mathcal{I} & \cdots \\
    \vdots & \vdots & \vdots & \vdots & \ddots \\
  \end{array}
\right)
\left(
  \begin{array}{c}
    I \\
    z I \\
    z^2 I \\
    z^3 I \\
    \cdots \\
  \end{array}
\right)
=
\left(
  \begin{array}{l}
    L_{00}^{(1)} \\
    L_{11}^{(1)}z \\
    L_{20}^{(1)}+L_{22}^{(1)}z^2 \\
    L_{31}^{(1)}z+L_{33}^{(1)}z^3 \\
    \vdots \\
  \end{array}
\right),
\end{align*}
which leads to the forms
\begin{align*}
p_{2j}(z)
&=
\mathcal{L}_{2j,0}^{(1)}+\mathcal{L}_{2j,2}^{(1)}z^2+\cdots+z^{2j},
\\\nonumber
p_{2j+1}(z)
&=
\mathcal{L}_{2j+1,1}^{(1)}z+\mathcal{L}_{2j+1,3}^{(1)}z^3+\cdots+z^{2j+1},
\quad
j=0,1,\cdots
\end{align*}
A similar expression for $\mathcal{Q}(\omega)$ is
\begin{align*}
q_{2j}(\omega)
&=
\mathcal{L}_{2j,0}^{(2)}+\mathcal{L}_{2j,2}^{(2)}\omega^2+\cdots+\omega^{2j},
\\\nonumber
q_{2j+1}(z)
&=
\mathcal{L}_{2j+1,1}^{(2)}\omega+\mathcal{L}_{2j+1,3}^{(2)}\omega^3+\cdots+\omega^{2j+1},
\quad
j=0,1,\cdots
\end{align*}
We note that while $p_{2j}(z)$ and $q_{2j}(\omega)$ contain only even powers of
$z$ and $\omega$ respectively,
$p_{2j+1}(z)$ and $q_{2j+1}(\omega)$ have only the odd powers.
This difference will be reflected throughout the manuscript, where
we derive results separately for the indices $2j$ and $2j+1$.
A fundamental reason for this, as will be observed, is that the matrix
$\mathcal{D}$ is no longer a diagonal but a block diagonal matrix
of $2\times 2$ blocks with matrix entries.

We begin with the following.
\begin{proposition}
The matrix polynomial sequences
$\{p_n(z)\}_{n=0}^{\infty}$ and $\{q_n(\omega)\}_{n=0}^{\infty}$
satisfy the relations
\begin{align}
\label{eqn: relations-biorthogonality between p and q}
\begin{split}
\langle p_{2j}, q_k\rangle &=d_{2j,2j+1}\delta_{2j+1,k},
\\
\langle p_{2j+1}, q_k\rangle &=d_{2j+1,2j}\delta_{2j,k},
\quad
j,k=0,1,2,\cdots
\end{split}
\end{align}
\end{proposition}
\begin{proof}
Using the properties of the bilinear form $\langle\cdot,\cdot\rangle$
and the definition
\eqref{eqn: definition of P and Q with L1-L2-X(z)-X(omega)}
we have
\begin{align*}
\langle p_i, q_j\rangle
&=
\langle \mathcal{P}_i(z),\mathcal{Q}_{j}(\omega)\rangle
=
\langle (\mathcal{L}_1\mathcal{X}(z))_{i},(\mathcal{L}_2\mathcal{X}(\omega))_{j}\rangle
\\
&=
\left(\mathcal{L}_1\langle \mathcal{X}(z),\mathcal{X}(\omega)\rangle(\mathcal{L}_2^{T}\right)_{ij}
=\mathcal{D}_{ij}
\end{align*}
since $\mathcal{M}=\mathcal{L}_1^{-1}\mathcal{D}\mathcal{L}_2^{-T}$.
The relations
\eqref{eqn: relations-biorthogonality between p and q}
follow from $\mathcal{D}$ being a block diagonal matrix with the
constituent blocks given by
\eqref{eqn: form of constituent matrices}.
\end{proof}
Similar relations hold for $q_n(\omega)$ too
\begin{align}
\label{eqn: relations-biorthogonality between q and p}
\langle p_{k}(z), q_{2j}(\omega)\rangle &=d_{2j+1,2j}\delta_{2j+1,k},
\\\nonumber
\langle p_{k}(z), q_{2j+1}(\omega)\rangle &=d_{2j,2j+1}\delta_{2j,k},
\quad
j,k=0,1,2,\cdots.
\end{align}
Here, we emphasise the relations
\eqref{eqn: relations-biorthogonality between p and q}
and
\eqref{eqn: relations-biorthogonality between q and p}
show that the sequence of matrix polynomials $\{p_i(z)\}$
and $\{q_i(z)\}$ are not biorthogonal
in the usual sense by which we mean
\begin{align*}
\langle p_{2j}(z),q_{2j}(z)\rangle=0
=
\langle p_{2j+1}(z),q_{2j+1}(z)\rangle,
\quad
j=0,1,\cdots.
\end{align*}
Another way to view this fact is to note that the above bilinear form
will always involve only the even powers of $z$ leading to even ordered
Gram entries which are zero by definition.
\subsection{Quasideterminant representation}
The theory of quasideterminants originated in attempts to define a determinant
for matrices with entries over a non-commutative ring
\cite{Quasideterminant-Gelfand-AdvMaths-2005}.
A quasideterminant is not actually an analogue of the commutative determinant,
but rather of the ratio of the determinant of an $n\times n$ matrix to the determinant
of an $(n-1)\times(n-1)$ sub-matrix.
This provides the crucial link when we want to find the matrix analogues of the
determinant representations of monic orthogonal polynomials with
scalar coefficients.

Let $\mathcal{V}$ be a $2\times2$ matrix with entries
$v_{ij}$, $i,j=1,2$, where $v_{ij}$ come from
a non-commutative ring.
The simplest example of quasideterminants of $\mathcal{V}$
are the following
\begin{align*}
|\mathcal{V}|_{11}&=v_{11}-v_{12}\cdot v_{22}^{-1}\cdot v_{21},
\quad
|\mathcal{V}|_{12}=v_{12}-v_{11}\cdot v_{21}^{-1}\cdot v_{22},
\\
|\mathcal{V}|_{21}&=v_{21}-v_{22}\cdot v_{12}^{-1}\cdot v_{11},
\quad
|\mathcal{V}|_{22}=v_{22}-v_{21}\cdot v_{11}^{-1}\cdot v_{12},
\end{align*}
which are also known as the Schur complements of the respective entries.
We will use the form $|\mathcal{V}|_{22}$ to find the quasideterminant
representations for $p_j(z)$ and $q_{j}(z)$
and use the notation $\theta_{\ast}$ to denote this fact, that is
\begin{align*}
\theta_{\ast}\mathcal{V}
=
\theta_{\ast}
\left(
  \begin{array}{cc}
    v_{11} & v_{12} \\
    v_{21} & v_{22} \\
  \end{array}
\right)
=
v_{22}-v_{21}v_{11}^{-1}v_{12}.
\end{align*}
The first step in obtaining the representations are the
orthogonality relations
\begin{align}
\label{eqn: relations-orthogonality for p-2j}
\langle p_{2j}(z), z^{k}\rangle
=
\left\{
  \begin{array}{ll}
    0, & \hbox{$k=0,1,\cdots,2j$;} \\
    d_{2j,2j+1}, & \hbox{$k=2j+1$.}
  \end{array}
\right.
\end{align}
\begin{align}
\label{eqn: relations-orthogonality for p-2j+1}
\langle p_{2j+1}(z), z^{k}\rangle
=
\left\{
  \begin{array}{ll}
    0, & \hbox{$k=0,1,\cdots,2j-1$;} \\
    d_{2j-1,2j}, & \hbox{$k=2j$.}
  \end{array}
\right.
\end{align}
that follow easily from
\eqref{eqn: relations-biorthogonality between p and q}.
We derive the quasideterminant representation for $p_{2j}(z)$ only.
The others follow similarly.
\begin{theorem}
Consider the Gram matrix
\begin{align*}
\mathcal{M}_{eo}^{[2j]}
=
\left(
  \begin{array}{cccc}
    m_{01} & m_{03} & \cdots & m_{0,2j-1} \\
    m_{21} & m_{23} & \cdots & m_{2,2j-1} \\
    \vdots & \vdots & \ddots & \vdots \\
    m_{2j-2,1} & m_{2j-2,3} & \cdots & m_{2j-2,2j-1} \\
  \end{array}
\right).
\end{align*}
Then, $p_{2j}(z)$, $j=0,1,\cdots$, has the quasideterminant representation
\begin{align}
\label{eqn: quasideterminant represtation for p-2j}
p_{2j}(z)=
\theta_{\ast}
\left(
  \begin{array}{cccc|c}
     &  & &  & \mathcal{I} \\
     &  &  & & z^2\mathcal{I} \\
     &  & \mathcal{M}_{eo}^{[2j]} & & \cdots \\
     &  &  & & z^{2j-2}\mathcal{I} \\\hline
    m_{2j,1} & m_{2j,3} &\cdots & m_{2j,2j-1} & z^{2j}\mathcal{I} \\
  \end{array}
\right).
\end{align}
\end{theorem}
\begin{proof}
Since
\begin{align*}
p_{2j}(z)
=\mathcal{L}_{2j,0}^{(1)}+\mathcal{L}_{2j,2}^{(1)}z^2+
\mathcal{L}_{2j,4}^{(1)}z^4+\cdots+z^{2j},
\end{align*}
so that the orthogonality relations
\eqref{eqn: relations-orthogonality for p-2j}
give
\begin{align*}
\mathcal{L}_{2j,0}^{(1)}\langle I, z^{k}\rangle+
\mathcal{L}_{2j,2}^{(1)}\langle z^2, z^{k}\rangle+
\mathcal{L}_{2j,4}^{(1)}\langle z^4, z^{k}\rangle+
\cdots+
\langle z^{2j}, z^{k}\rangle=0,
\end{align*}
for $k=1,3,\cdots,2j-1$, which leads to
the system of equations represented as
\begin{align*}
\left(
  \begin{array}{cccc}
    \mathcal{L}_{2j,0}^{(1)} & \mathcal{L}_{2j,2}^{(1)}  & \cdots & \mathcal{L}_{2j,2j-2}^{(1)} \\
  \end{array}
\right)
\mathcal{M}_{eo}^{[2j]}
=-
\left(
  \begin{array}{cccc}
    m_{2j,1} & m_{2j,3} & \cdots & m_{2j,2j-1} \\
  \end{array}
\right).
\end{align*}
Since
\begin{align*}
p_{2j}(z)
&=
z^{2j}+\left(
  \begin{array}{ccc}
    \mathcal{L}_{2j,0}^{(1)}  & \cdots & \mathcal{L}_{2j,2j-2}^{(1)} \\
  \end{array}
\right)
\left(
  \begin{array}{c}
    I \\
  \vdots \\
    z^{2j-2}I \\
  \end{array}
\right),
\end{align*}
we have the quasideterminant representation
\eqref{eqn: quasideterminant represtation for p-2j}.
\end{proof}
Consider another Gram matrix
\begin{align*}
\mathcal{M}_{oe}^{[2j]}
=
\left(
  \begin{array}{cccc}
    m_{10} & m_{12} & \cdots & m_{1,2j-2} \\
    m_{30} & m_{32} & \cdots & m_{3,2j-2} \\
    \vdots & \vdots & \ddots & \vdots \\
    m_{2j-1,0} & m_{2j-1,2} & \cdots & m_{2j-1,2j-2} \\
  \end{array}
\right).
\end{align*}
Then, we have the following quasideterminant representations
for $j=0,1,\cdots$,
\begin{align*}
p_{2j+1}(z)=
\theta_{\ast}
\left(
  \begin{array}{cccc|c}
     &  & &  & z\mathcal{I} \\
     &  &  & & z^3\mathcal{I} \\
     &  & \mathcal{M}_{oe}^{[2j]} & & \cdots \\
     &  &  & & z^{2j-1}\mathcal{I} \\\hline
    m_{2j+1,0} & m_{2j+1,2} &\cdots & m_{2j+1,2j-2} & z^{2j+1}\mathcal{I} \\
  \end{array}
\right),
\end{align*}
Similarly, from \eqref{eqn: relations-biorthogonality between q and p}
we have for $j=0,1,\cdots$
\begin{align*}
q_{2j}(z)=
\theta_{\ast}
\left(
  \begin{array}{cccc|c}
     &  & &  & m_{1,2j} \\
     &  &  & & m_{3,2j} \\
     &  & \mathcal{M}_{oe}^{[2j]} & & \cdots \\
     &  &  & & m_{2j-1,2j} \\\hline
    \mathcal{I} & z^2\mathcal{I} &\cdots & z^{2j-2}\mathcal{I} & z^{2j}\mathcal{I} \\
  \end{array}
\right),
\end{align*}
\begin{align*}
q_{2j+1}(z)=
\theta_{\ast}
\left(
  \begin{array}{cccc|c}
     &  & &  & m_{0,2j+1} \\
     &  &  & & m_{2,2j+1} \\
     &  & \mathcal{M}_{eo}^{[2j]} & & \cdots \\
     &  &  & & m_{2j-2,2j+1} \\\hline
    z\mathcal{I} & z^3\mathcal{I} &\cdots & z^{2j-1}\mathcal{I} & z^{2j+1}\mathcal{I} \\
  \end{array}
\right),
\end{align*}
In the remaining portion of the manuscript, we will be
concerned more with the non-zero entries of all the matrices involved
rather than the actual expressions of these entries. Our focus
will be on structural relations between the various entities that we derive.
\section{The Christoffel transformation}
\label{sec: Christoffel transformation}
The underlying theme while discussing
Christoffel transformation is that it
should reflect the fact that
$p_{2j}(z)$ and $p_{2j+1}(z)$
are different as polynomials.
This is achieved by truncating
after $2k-1$ rows and columns so that we obtain matrices of
order $2k\times 2k$, thereby preserving the
$2\times 2$ block structure.

If we observe, the analysis so far has its origins in the single most important fact that
the main diagonal (and hence every alternate diagonal due to the symmetry
associated with the quadratic transformation $z\mapsto z^2$) of the Gram matrix
is the zero diagonal.
Hence, a desirable starting point could also have been the Gram matrix
$\hat{\mathcal{M}}$ defined as
\begin{align*}
\hat{\mathcal{M}}
=
\Lambda\mathcal{M}
=
\left(
  \begin{array}{ccccc}
    m_{10} & 0 & m_{12} & 0 & \cdots \\
    0 & m_{21} & 0 & m_{23} & \cdots \\
    m_{30} & 0 & m_{32} & 0 & \cdots \\
    0 & m_{41} & 0 & m_{43} & \cdots \\
    \vdots & \vdots & \vdots & \vdots & \ddots \\
  \end{array}
\right),
\quad
\Lambda=
\left(
  \begin{array}{cccc}
    0 & I & 0 & \cdots \\
    0 & 0 & I & \cdots \\
    \vdots & \vdots & \vdots & \ddots \\
  \end{array}
\right),
\end{align*}
and then proceeding with the
$\hat{\mathcal{L}}_1^{-1}\hat{\mathcal{D}}\hat{\mathcal{L}}_2^{-T}$
factorization. It turns out that
$\hat{\mathcal{L}}_1$ and
$\hat{\mathcal{L}}_2$ have the same structure as
$\mathcal{L}_1$ and $\mathcal{L}_2$ while
$\hat{\mathcal{D}}$ is a diagonal matrix.
In fact given any Gram matrix $\mathcal{G}$, the transformation
$\mathcal{G}\mapsto\hat{\mathcal{G}}:=\Lambda\mathcal{G}$
is known as the Christoffel transformation of the
Gram matrix $\mathcal{G}$ and is well studied
\cite{Manas-CD-MOP-JMAA-2014, Paco-Christoffel-MOP-2017}.
We derive some of the existing results for $z$
\cite{Manas-CD-LU-Survey-2021}
in the present case of
quadratic transformation $z\mapsto z^2$.

Since $\hat{\mathcal{M}}=\Lambda\mathcal{M}$, we have
\begin{align*}
\hat{\mathcal{L}}_1^{-1}\hat{\mathcal{D}}\hat{\mathcal{L}}_2^{-T}
=
\Lambda
\mathcal{L}_1^{-1}\mathcal{D}\mathcal{L}_2^{-T}
\Longrightarrow
\hat{\mathcal{D}}\hat{\mathcal{L}}_2^{-T}
=
\sigma
\mathcal{D}\mathcal{L}_2^{-T};
\quad
\sigma:=\hat{\mathcal{L}}_1\Lambda\mathcal{L}_1^{-1}.
\end{align*}
We call $\sigma$
the connector which is also given by the expression
$\sigma=\hat{\mathcal{D}}\hat{\mathcal{L}}_2^{-T}\mathcal{L}_2^{T}\mathcal{D}^{-1}$.
A key role played by $\sigma$, which somewhat explains the
use of the term connector, is brought out in the following result.
\begin{proposition}
Given the matrix polynomial sequence
$\{p_n(z)\}_{n=0}^{\infty}$
associated with the Gram matrix $\mathcal{M}$,
let $\{\hat{p}_n(z)\}_{n=0}^{\infty}$
be associated with the Christoffel transformation
$\hat{\mathcal{M}}$ of $\mathcal{M}$.
Then, we have
\begin{align}
\label{eqn: Christoffel-form of polynomials p-n}
\begin{split}
\hat{p}_{2j}
&=\frac{1}{z}p_{2j+1}(z),
\\
\hat{p}_{2j+1}
&=
\frac{1}{z}
\left(
p_{2j+2}(z)-p_{2j+2}(0)p_{2j}(0)^{-1}p_{2j}(z)
\right),
\quad
j=0,1,\cdots.
\end{split}
\end{align}
\end{proposition}
\begin{proof}
Since $\mathcal{P}(z)=\mathcal{L}_1\mathcal{X}(z)$, we have
\begin{align*}
\sigma\mathcal{P}(z)
=
\hat{D}\hat{\mathcal{L}}_2^{-T}\mathcal{L}_2^{T}\mathcal{D}^{-1}\mathcal{L}_1\chi(z)
=
\hat{D}\hat{\mathcal{L}}_2^{-T}\hat{\mathcal{M}}^{-1}\Lambda\mathcal{X}(z)
=
z\hat{\mathcal{P}}(z).
\end{align*}
Further, comparing both the expressions for $\sigma$, we conclude that $\sigma$
has to be necessarily of the form
\begin{align*}
\sigma=
\left(
  \begin{array}{cccccc}
    0 & I & 0 & 0 & 0 & \cdots \\
    \sigma_{10} & 0 & I & 0 & 0 & \cdots \\
    0 & 0 & 0 & I & 0 & \cdots \\
    0 & 0 & \sigma_{32} & 0 & I & \cdots \\
    0 & 0 & 0 & 0 & 0 & \cdots \\
    \vdots & \vdots & \vdots & \vdots & \vdots & \ddots \\
  \end{array}
\right),
\quad
\sigma_{2j+1,2j}=\hat{d}_{2j+1,2j+1}d_{2j,2j+1}^{-1},
\quad
j=0,1,\cdots
\end{align*}
Then from $\sigma\mathcal{P}(z)=z\hat{\mathcal{P}}(z)$,
we obtain
\begin{align*}
p_{2j+2}(z)+\sigma_{2j+1,2j}p_{2j}(z)=z\hat{p}_{2j+1}(z),
\quad
j=0,1,\cdots.
\end{align*}
The value of $\sigma_{2j+1,2j}$ is found by substituting $z=0$
in the above relation leading to
\eqref{eqn: Christoffel-form of polynomials p-n}.
\end{proof}
We obtain in similar fashion
\begin{align}
\label{eqn: relation between q and q-hat}
\mathcal{Q}(\omega)=\mathcal{L}_2\mathcal{X}(\omega)
\Longrightarrow
\mathcal{Q}^{T}(\omega)\mathcal{D}^{-1}=
\hat{\mathcal{Q}}^{T}(\omega)\hat{\mathcal{D}}^{-1}\sigma
\end{align}
leading to structural relations between $q_k(z)$ and $\hat{q}_l(z)$.
However, because of the appearance of the expressions
$\mathcal{Q}^{T}(\omega)$ and $\hat{\mathcal{Q}}^{T}(\omega)$,
\eqref{eqn: relation between q and q-hat}
serves as the motivation to define the kernel polynomials related to the Gram
matrices $\mathcal{M}$ and $\hat{\mathcal{M}}$.
\subsection{Christoffel-Darboux kernels}
Let us first work out a special case that will serve as
motivation for the definitions.
Consider the following
\begin{align*}
\left(
  \begin{array}{cccccc}
    \mathcal{L}_{00}^{(1)} & 0 & 0 & 0  \\
    0 & \mathcal{L}_{11}^{(1)} & 0 & 0  \\
    \mathcal{L}_{20}^{(1)} & 0 & \mathcal{L}_{22}^{(1)} & 0  \\
    0 & \mathcal{L}_{31}^{(1)} & 0 & \mathcal{L}_{33}^{(1)} \\
   \end{array}
\right)
\left(
  \begin{array}{cc|cc}
    1 & 0 & 0 & 0  \\
    0 & 0 & 0 & 0  \\\hline
    0 & 0 & 1 & 0  \\
    0 & 0 & 0 & 0  \\
   \end{array}
\right)
\left(
  \begin{array}{c}
    I \\
   z I \\
   z^2I \\
   z^3I \\
  \end{array}
\right)
=
\left(
  \begin{array}{c}
    p_0(z) \\
    0 \\
    p_2(z) \\
    0 \\
  \end{array}
\right),
\end{align*}
which is to pick only $p_{2j}(z)$.
Denoting the right hand most vector as
$\mathcal{P}_e^{[4]}(z)$
we define the matrix kernel polynomial
$\mathcal{K}^{[2]}(z,\omega)$ as
\begin{align*}
\mathcal{K}^{[2]}(z,\omega)=
[\mathcal{Q}^{[4]}(\omega)]^{T}
[\mathcal{D}^{[4]}]^{-1}
\mathcal{P}_e^{[4]}(z)
=
\sum_{j=0}^{1}q_{2j+1}^{T}(\omega)d_{2j,2j+1}^{-1}p_{2j}(z).
\end{align*}
Hence, we give the following
\begin{definition}
\label{definition: kernel polynomial}
Given the biorthogonal polynomial sequences
$\{\mathcal{P}_n(z)\}_{n=0}^{\infty}$ and
$\{\mathcal{Q}_n(z)\}_{n=0}^{\infty}$ that arise from a
Gram matrix  with even ordered moments as zero,
the associated kernel polynomials are defined as
\begin{align}
\label{eqn: kernel polynomials from M}
\begin{split}
\mathcal{K}^{[2n]}(z,\omega)
&=
\sum_{j=0}^{n}q_{2j+1}^{T}(\omega)d_{2j,2j+1}^{-1}p_{2j}(z),
\\
\mathcal{K}^{[2n+1]}(z,\omega)
&=
\sum_{j=0}^{n}q_{2j}^{T}(\omega)d_{2j+1,2j}^{-1}p_{2j+1}(z),
\quad
n=0,1,\cdots.
\end{split}
\end{align}
\end{definition}
Next the left hand side of the equality
$\mathcal{Q}^{T}(\omega)\mathcal{D}^{-1}=
\hat{\mathcal{Q}}^{T}(\omega)\hat{\mathcal{D}}^{-1}\sigma$
as obtained in \eqref{eqn: relation between q and q-hat}
gives
\begin{align*}
\left(
  \begin{array}{cccc}
    q_0^{T}(\omega) & q_1^{T}(\omega) & q_{2}^{T}(\omega) & q_{3}^{T}(\omega)\\
  \end{array}
\right)
\left(
  \begin{array}{cc|cc}
    0 & d_{10}^{-1} & 0 & 0  \\
    d_{01}^{-1} & 0 & 0 & 0 \\\hline
    0 & 0 & 0 & d_{32}^{-1} \\
    0 & 0 & d_{23}^{-1} & 0 \\
   \end{array}
\right)
\left(
  \begin{array}{c}
    p_0(z) \\
    0 \\
    p_2(z) \\
    0 \\
  \end{array}
\right)
\end{align*}
which is $\mathcal{K}^{[2]}(z,\omega)$.
Similarly, the right hand side of the equality gives
\begin{align*}
\lefteqn{
\left(
  \begin{array}{cccc}
    \hat{q}_0^{T}(\omega) & \hat{q}_1^{T}(\omega) & \hat{q}_{2}^{T}(\omega) & \hat{q}_{3}^{T}(\omega)\\
  \end{array}
\right)
}
\\
&&\times
\left(
  \begin{array}{cc|cc}
    \hat{d}_{00}^{-1} & 0 & 0 & 0  \\
    0 & \hat{d}_{11}^{-1} & 0 & 0 \\\hline
    0 & 0 & \hat{d}_{22}^{-1} & 0 \\
    0 & 0 & 0 & \hat{d}_{33}^{-1} \\
   \end{array}
\right)
\left(
  \begin{array}{cc|cc}
    0 & I & 0 & 0 \\
    \sigma_{10} & 0 & I & 0 \\\hline
    0 & 0 & 0 & I \\
    0 & 0 & \sigma_{32} & 0 \\
  \end{array}
\right)
\left(
  \begin{array}{c}
    p_0(z) \\
    0 \\
    p_2(z) \\
    0 \\
  \end{array}
\right)
\end{align*}
which gives
\begin{align*}
z\sum_{j=0}^{1}\hat{q}_{2j+1}^{T}(\omega)\hat{d}_{2j+1,2j+1}^{-1}\hat{p}_{2j+1}(\omega)
-
\hat{q}_{3}(\omega)\hat{d}_{33}^{-1}p_{3}(z).
\end{align*}
Hence we give the following
\begin{definition}
Given the Gram matrix $\mathcal{M}$, let
$\{\hat{p}_n(z)\}_{n=0}^{\infty}$
and
$\{\hat{q}_n(z)\}_{n=0}^{\infty}$
be the biorthogonal polynomial sequences
associated with the Christoffel transformation
$\hat{\mathcal{M}}$ of $\mathcal{M}$.
Then, the associated kernel polynomials are defined as
\begin{align}
\label{eqn: kernel polynomials from M-hat}
\begin{split}
\hat{\mathcal{K}}^{[2n]}(z,\omega)
&=
\sum_{j=0}^{n}\hat{q}_{2j}^{T}(\omega)\hat{d}_{2j,2j}^{-1}\hat{p}_{2j}(z),
\\
\hat{\mathcal{K}}^{[2n+1]}(z,\omega)
&=
\sum_{j=0}^{n}\hat{q}_{2j+1}^{T}(\omega)\hat{d}_{2j+1,2j+1}^{-1}\hat{p}_{2j+1}(z),
\quad
n=0,1,\cdots.
\end{split}
\end{align}
 \end{definition}
The preceding discussion immediately gives
\begin{proposition}
The kernel polynomials $\mathcal{K}^{[j]}(z,\omega)$ and
$\hat{\mathcal{K}}^{[j]}(z,\omega)$
associated, respectively, with $\mathcal{M}$ and $\hat{\mathcal{M}}$
are related as
\begin{align}
\label{eqn: relation between kernel polynomials}
\begin{split}
\mathcal{K}^{[2n]}(z,\omega)
&=
z\hat{\mathcal{K}}^{[2n+1]}(z,\omega)-
\hat{q}_{2n+1}(\omega)\hat{d}_{2n+1,2n+1}^{-1}p_{2n+2}(z),
\\
\mathcal{K}^{[2n+1]}(z,\omega)
&=
z\hat{\mathcal{K}}^{[2n]}(z,\omega),
\quad
n=0,1,\cdots.
\end{split}
\end{align}
\end{proposition}
\begin{proof}
The first relation follows from extending the particular case
above to matrices truncated after (2n+1) rows and columns.
The second relation follows on the same lines as the first one
except for the fact that we replace 
$(p_0(z), 0, p_{2}(z), 0, \cdots, p_{2n}(z), 0)^{T}$
with the vector
$(0, p_1(z), 0, p_{3}(z), \cdots,0, p_{2n+1}(z))^{T}$.
\end{proof}
Next, we derive the ABC Theorems for both $\mathcal{M}$
and $\hat{\mathcal{M}}$. We will use the following three matrices
\begin{align*}
\Theta^{[2n]}
&=
\left(
  \begin{array}{ccccc}
    0 & 0 & \cdots & 0 & \mathcal{I} \\
    \mathcal{I} & 0 & \cdots & 0 & 0 \\
    \vdots & \ddots & \ddots & \vdots & \vdots \\
    0 & 0 & \cdots & 0 & 0 \\
    0 & 0 & \cdots & \mathcal{I} & 0  \\
  \end{array}
\right),
\quad
(\Theta^{[2n]})^{-1}=(\Theta^{[2n]})^{T},
\\
\Pi_e
&=
\left(
   \begin{array}{ccccc}
     \vec{e}_0 & \vec{0} & \vec{e}_2 & \vec{0} & \cdots \\
     \downarrow & \downarrow & \downarrow & \downarrow & \cdots \\
   \end{array}
\right)
\quad\mbox{and}\quad
\Pi_o
=
\left(
   \begin{array}{ccccc}
     \vec{0} & \vec{e}_1 & \vec{0} & \vec{e}_3 & \cdots \\
     \downarrow & \downarrow & \downarrow & \downarrow & \cdots \\
   \end{array}
\right),
 \end{align*}
where $\vec{e}_i=(0,\cdots,i,\cdots,0)$, $i=0,1,\cdots$ are the canonical vectors
and $\vec{0}$ is the zero vector. The arrows indicate that the $(\vec{e}_i)'s$ are written
as columns in $\Pi_e$ and $\Pi_o$.
\begin{theorem}
[ABC Theorem]
\label{theorem: theorem ABC for M}
The matrix kernel polynomials associated with the
Gram matrix $\mathcal{M}$ have the representation
\begin{align*}
\mathcal{K}^{[2n]}(z,\omega)
&=(\mathcal{X}^{[2n]}(\omega))^{T}\Pi_o^{[2n]}(\mathcal{M}_e^{[2n]})^{-1}
\Pi_e^{[2n]}\mathcal{X}^{[2n]}(z),
\\
\mathcal{K}^{[2n+1]}(z,\omega)
&=(\mathcal{X}^{[2n]}(\omega))^{T}\Pi_e^{[2n]}(\mathcal{M}_o^{[2n]})^{-1}
\Pi_o^{[2n]}\mathcal{X}^{[2n]}(z),
\end{align*}
where $\mathcal{M}_e^{[2n]}=
(\mathcal{L}_1^{[2n]})^{-1}\mathcal{D}_e^{[2n]}
(\mathcal{L}_2^{[2n]})^{-T}$
and
$\mathcal{M}_o^{[2n]}=
(\mathcal{L}_1^{[2n]})^{-1}\mathcal{D}_o^{[2n]}
(\mathcal{L}_2^{[2n]})^{-T}$
with
\begin{align*}
\mathcal{D}_e^{[2n]}
&=
{\rm diag}~(d_{0,1},d_{1,0},\cdots,d_{2n,2n+1},d_{2n+1,2n})\Theta^{[2n]},
\\
\mathcal{D}_o^{[2n]}
&=
\Theta^{[2n]}{\rm diag}~(d_{1,0},d_{0,1},\cdots,d_{2n+1,2n},d_{2n,2n+1}).
\end{align*}
\end{theorem}
\begin{proof}
We prove the expression only for $\mathcal{K}^{[2n]}(z,\omega)$.
Along with the following relations
\begin{align*}
\mathcal{L}_1^{[2n]}\Pi_e^{[2n]}\mathcal{X}^{[2n]}(z)
&=
\left(
  \begin{array}{ccccc}
    p_0(z) & 0 & \cdots & p_{2n}(z) & 0 \\
  \end{array}
\right)^{T},
\\
\mathcal{L}_2^{[2n]}\Pi_0^{[2n]}\mathcal{X}^{[2n]}(\omega)
&=
\left(
  \begin{array}{ccccc}
    0 & q_{1}(\omega) & \cdots & 0 & q_{2n+1}(\omega) \\
  \end{array}
\right)^{T},
\end{align*}
we note that the kernel polynomial $\mathcal{K}^{[2n]}(z,\omega)$ contains
terms involving $q_{2j+1}(\omega)$ and $p_{2j}(z)$.
Hence, we need to make one
more transformation that interchanges $0$ and $q_{2j+1}(\omega)$,
which we do by post-multiplying
$\mathcal{L}_2^{[2n]}\Pi_0^{[2n]}\mathcal{X}^{[2n]}(\omega)$
with the matrix $\Theta^{[2n]}$.
Since
$
\mathcal{K}^{[2n]}(z,\omega)
=
[\mathcal{L}_2^{[2n]}\Pi_o^{[2n]}\mathcal{X}(\omega)]^{T}
[\Theta^{[2n]}]^{T}
(\mathcal{D}^{[2n]})^{-1}
\mathcal{L}_1^{[2n]}\Pi_e^{[2n]}\mathcal{X}^{[2n]}(z),
$
the expression for $\mathcal{K}^{[2n]}(z,\omega)$ follows.
\end{proof}
%
%
%
%
\section{Illustration: Hankel symmetry}
\label{sec: Hankel symmetry}
We begin this section with the question:
when does the Gram matrix $\mathcal{M}$
given by
\eqref{eqn: definition of Gram matrix with even zero}
have a factorization of the form
$\mathcal{M}=\mathcal{L}\mathcal{D}\mathcal{L}^{T}$?
We consider the part of the matrix $\mathcal{M}$ above the diagonal
%
and look for a map that transforms the condensed matrix
$\tilde{\mathcal{M}}$ to the block structured matrix $\mathcal{M}$,
$\tilde{\mathcal{M}}\rightsquigarrow\mathcal{M}$,
as given below
\begin{align*}
\left(
  \begin{array}{ccc}
    m_{0,1} & m_{0,3} & \cdots \\
    m_{1,2} & m_{1,4} & \cdots \\
    \vdots & \vdots  & \ddots \\
  \end{array}
\right)
\rightsquigarrow
\left(
  \begin{array}{cc|cc|c}
    0 & m_{01} & 0 & m_{03}  & \cdots \\
    m_{10} & 0 & m_{12} & 0  & \cdots \\\hline
    0 & m_{21} & 0 & m_{23}  & \cdots \\
    m_{30} & 0 & m_{32} & 0  & \cdots \\\hline
   \vdots & \vdots & \vdots  & \vdots & \ddots \\
  \end{array}
\right).
\end{align*}
If we impose Hankel symmetry so that $m_{i,j}=m_{k,l}$
if $i+j=k+l$, we may interpret the above map as
each entry $m_{i,j}$ being mapped to a block matrix on
the right, which implies that
we need to use the Kronecker product
\begin{align*}
\tilde{\mathcal{M}}\otimes\mathcal{J}_2,
\quad
\mbox{where}
\quad
\mathcal{J}_2=\left(
                \begin{array}{cc}
                  0 & I \\
                  I & 0 \\
                \end{array}
              \right).
\end{align*}
So, if $\tilde{\mathcal{M}}=\tilde{\mathcal{L}}
\tilde{\mathcal{D}}\tilde{\mathcal{L}}^{T}$,
then we have
$
\tilde{\mathcal{M}}\otimes\mathcal{J}_2
=
\tilde{\mathcal{L}}\tilde{\mathcal{D}}\tilde{\mathcal{L}}^{T}\otimes\mathcal{J}_2
$.
The use of Kronecker product along with the following properties
\begin{align}
\label{eqn: Kronecker relations}
\begin{split}
(A\otimes B)(C\otimes D)
&=AC\otimes BD,
\quad
(A\otimes B)^{T}=A^{T}\otimes B^{T},
\\
(A\otimes B)^{-1}
&=A^{-1}\otimes B^{-1}
\end{split}
\end{align}
simplifies the transition to the block structure to a great extent and
we get
\begin{align*}
(\mathcal{L}\otimes I)(\mathcal D\otimes\mathcal{J}_2)(\mathcal{L}^{T}\otimes I)=
(\mathcal{L}\mathcal{D}\otimes\mathcal{J}_2)(\mathcal{L}^{T}\otimes I)=
\mathcal{L}\mathcal{D}\mathcal{L}^{T}\otimes\mathcal{J}_2.
\end{align*}
Consequently we obtain the following
\begin{align*}
\tilde{M}=\tilde{\mathcal{L}}\tilde{\mathcal{D}}\tilde{\mathcal{L}}^{T}
\Longrightarrow
\mathcal{M}=
(\tilde{\mathcal{L}}\otimes I)
(\tilde{\mathcal{D}}\otimes\mathcal{J}_2)
(\tilde{\mathcal{L}}^{T}\otimes I)
=
\mathcal{L}_1^{-1}\mathcal{D}\mathcal{L}_2^{-T},
\end{align*}
which provides the relation between the respective factorizations.
Further, with \eqref{eqn: Kronecker relations} if we let
$\mathcal{L}_1^{-1}=\tilde{\mathcal{L}}\otimes \mathcal{I}$
and $\mathcal{L}_2^{-T}=(\tilde{\mathcal{L}}^{T}\otimes I)$,
then
$\mathcal{L}_1=\mathcal{L}_2=\tilde{\mathcal{L}}^{-1}\otimes \mathcal{I}$.
The polynomial sequence $\{p_n(z)\}_{n=0}^{\infty}$
is generated from
$\mathcal{L}_1\mathcal{X}(z)=(\tilde{\mathcal{L}}^{-1}\otimes \mathcal{I})\mathcal{X}(z)$
as
\begin{align*}
\left(
  \begin{array}{cc|cc|c}
    \tilde{\mathcal{L}}_{00} & 0  & 0 & 0  & \cdots \\
    0 & \tilde{\mathcal{L}}_{00} & 0 & 0  & \cdots \\\hline
    \tilde{\mathcal{L}}_{10} & 0 & \tilde{\mathcal{L}}_{11} & 0  & \cdots \\
    0 & \tilde{\mathcal{L}}_{10} & 0 & \tilde{\mathcal{L}}_{11}  & \cdots \\\hline
    \vdots & \vdots & \vdots & \vdots  & \ddots \\
  \end{array}
\right)
\left(
  \begin{array}{c}
    I \\
    zI \\
    z^2I \\
    z^3I \\
  \vdots \\
  \end{array}
\right)
=
\left(
  \begin{array}{c}
    p_0(z) \\
    p_1(z )\\
    p_2(z ) \\
    p_3(z ) \\
    \vdots \\
  \end{array}
\right),
\end{align*}
so that we have the forms
\begin{align*}
p_{2j}(z)
&=
\tilde{\mathcal{L}}_{j,j}z^{2j}+\tilde{\mathcal{L}}_{j,j-1}z^{2j-2}+\cdots+
\tilde{\mathcal{L}}_{j,1}z^2+\tilde{\mathcal{L}}_{j,0},
\\
p_{2j+1}(z)
&=
zp_{2j}(z),
\quad
j=0,1,\cdots.
\end{align*}
Hence, if we assume that the condensed matrix $\tilde{\mathcal{M}}$
is positive-definite, or in other words its entries are the moments
coming from a determined matrix Hamburger moment problem, the
factorization
$\tilde{\mathcal{M}}=\tilde{\mathcal{L}}\tilde{\mathcal{D}}\tilde{\mathcal{U}}$
exists and all of the above results follow through.
We also note that the matrices
$\mathcal{M}_{eo}^{[2j]}$ and $\mathcal{M}_{oe}^{[2j]}$
used in the quasi-determinant representations of $p_{n}(z)$ and $q_{n}(z)$
are nothing but truncations of $\tilde{\mathcal{M}}$ and hence are invertible.

This also bring us to the following observation
that because of the underlying Hankel symmetry, we have
$q_j(z)=p_j(z)$, $j=0,1,\cdots$, where $q_j(z)$ is obtained from
$\mathcal{L}_2\mathcal{X}(z)$.
Further, if we denote
$\tilde{\mathcal{D}}=
\mbox{diag}~(\tilde{d}_{0,0},\, \tilde{d}_{1,1},\,\cdots)$,
then $d_{2j,2j+1}=d_{2j+1,2j}=\tilde{d}_{jj}$, $j=0,1,\cdots$.
The biorthogonal relations
\eqref{eqn: relations-biorthogonality between p and q}
immediately yield
\begin{align}
\label{eqn: relations-biorthogonality between p and q in Hankel case}
\begin{split}
\langle p_{2j}, p_k\rangle &=d_{j,j}\delta_{2j+1,k},
\\
\langle p_{2j+1}, p_k\rangle &=d_{j,j}\delta_{2j,k},
\quad
j,k=0,1,2,\cdots,
\end{split}
\end{align}
which shows that $\{p_n(z)\}_{n=0}^{\infty}$ is a polynomial sequence
that is biorthogonal to itself with respect to the bilinear form
$\langle\cdot,\cdot\rangle$, which is, thus, not quasi-definite.
In the scalar case, the relations
\eqref{eqn: relations-biorthogonality between p and q in Hankel case}
exist and such systems are called almost orthogonal, arising in works related to
indefinite analogues of the Hamburger and Stieltjes moment problems.
We refer the reader to
\cite[Sections~2, 3]{Maxim-polynomial-mapping-JAT-2013}
for necessary references in this direction
and a view of the unwrapping of measures via
continued fractions.
Further, some of the results in the present section
that are reduced to special forms due to Hankel symmetry
are also presented
\cite{Maxim-polynomial-mapping-JAT-2013},
though in the scalar case,
where the approach is through the $\mathcal{L}\mathcal{U}$
decomposition of the underlying Jacobi matrices.

We end with the special forms of the kernel polynomials and the ABC Theorem
in case of Hankel symmetry.
From \eqref{eqn: kernel polynomials from M},
the kernel polynomials associated to
Gram matrix $\mathcal{M}$ are given by
\begin{align}
\label{eqn: kernel polynomials-Hankel symmetry}
\begin{split}
\mathcal{K}^{[2n]}(z,\omega)
&=
\omega\sum_{j=0}^{n}p_{2j}^{T}(\omega)\tilde{d}_{j,j}^{-1}p_{2j}(z),
\\
\mathcal{K}^{[2n+1]}(z,\omega)
&=
z\sum_{j=0}^{n}p_{2j}^{T}(\omega)\tilde{d}_{j,j}^{-1}p_{2j}(z),
\quad
n=0,1,\cdots,
\end{split}
\end{align}
having representations as obtained in
Theorem~\ref{theorem: theorem ABC for M}.
We only note that $\mathcal{M}_e^{[2n]}$ and
$\mathcal{M}_o^{[2n]}$ are still not presented in the form of a
$\mathcal{L}\mathcal{D}\mathcal{U}$ decomposition since
$\mathcal{D}_{e}^{[2n]}$ and $\mathcal{D}_{o}^{[2n]}$
are not diagonal.
This is perhaps because the above forms
\eqref{eqn: kernel polynomials-Hankel symmetry},
even in case of Hankel symmetry, appear as
forward shifts in $\omega$ and $z$ respectively and
needs to be investigated further.


\begin{thebibliography}{99}
%
\bibitem{Adler-group-factorization-1997}
M. Adler\ and\ P. van Moerbeke,
Group factorization, moment matrices,
and Toda lattices,
Internat. Math. Res. Notices
{\bf 1997}, no.~12, 555--572.
%
\bibitem{Adler-momet-matrices-2009}
M. Adler, P. van Moerbeke\ and\ P. Vanhaecke,
Moment matrices and multi-component KP,
with applications to random matrix theory,
Comm. Math. Phys.
{\bf 286} (2009), no.~1, 1--38.
%
\bibitem{Akhiezer-book}
N. I. Akhiezer,
{\it The classical moment problem and some
related questions in analysis},
translated by N. Kemmer,
Hafner Publishing Co., New York, 1965.
%
\bibitem{Manas-Ulysess-multiple-Gauss-Borel-2011}
C. \'{A}lvarez-Fern\'{a}ndez, U. Fidalgo Prieto\ and\ M. Ma\~{n}as,
Multiple orthogonal polynomials of mixed type:
Gauss-Borel factorization and the multi-component
2D Toda hierarchy,
Adv. Math.
{\bf 227} (2011), no.~4, 1451--1525.
%
\bibitem{Manas-CD-LU-Survey-2021}
C. \'{A}lvarez-Fern\'{a}ndez\ and\ M. Ma\~{n}as,
Chapter in Orthogonal polynomials: current trends and applications.
Edited by Francisco Marcellán and Edmundo J. Huertas.
SEMA SIMAI Springer Series, 22.
Springer, Cham, [2021], ©2021. 327 pp.
%
\bibitem{Paco-Christoffel-MOP-2017}
C. \'{A}lvarez-Fern\'{a}ndez\ et al.,
Christoffel transformations for matrix
orthogonal polynomials in the real line and the non-Abelian
2D Toda lattice hierarchy,
Int. Math. Res. Not. IMRN {\bf 2017}, no.~5, 1285--1341.
%
\bibitem{Manas-CD-MOP-JMAA-2014}
C. \'{A}lvarez-Fern\'{a}ndez\ and\ M. Ma\~{n}as,
On the Christoffel-Darboux formula for generalized
matrix orthogonal polynomials,
J. Math. Anal. Appl.
{\bf 418} (2014), no.~1, 238--247.
%
\bibitem{Manas-multivariate-AdvM-2016}
G. Ariznabarreta\ and\ M. Ma\~{n}as,
Multivariate orthogonal polynomials and integrable systems,
Adv. Math.
{\bf 302} (2016), 628--739.
%
\bibitem{Bessis-Moussa-iterated-mapping-CommMathPhy-1983}
D. Bessis\ and\ P. Moussa,
Orthogonality properties of iterated polynomial mappings,
Comm. Math. Phys.
{\bf 88} (1983), no.~4, 503--529.
%
\bibitem{Ismail-Charris-Sieved-VII-TAMS-1993}
J. A. Charris\ and\ M. E. H. Ismail,
Sieved orthogonal polynomials. VII.
Generalized polynomial mappings,
Trans. Amer. Math. Soc.
{\bf 340} (1993), no.~1, 71--93.
%
\bibitem{Chen-matrixCF-LAA-1998}
G. Chen\ and\ Y. Hu,
The truncated Hamburger matrix moment problems
in the nondegenerate and degenerate cases, and
matrix continued fractions,
Linear Algebra Appl.
{\bf 277} (1998), no.~1-3, 199--236.
%
%
%
\bibitem{Maxim-polynomial-mapping-JAT-2013}
M. Derevyagin,
On the relation between Darboux transformations
and polynomial mappings,
J. Approx. Theory {\bf 172} (2013), 4--22.
%
%
\bibitem{Quasideterminant-Gelfand-AdvMaths-2005}
I. Gelfand\ et al.,
Quasideterminants,
Adv. Math.
{\bf 193} (2005), no.~1, 56--141.
%
\bibitem{Geronimo-Van-Assche-several-interval-TAMS-1988}
J. S. Geronimo\ and\ W. Van Assche,
Orthogonal polynomials on several intervals via a
polynomial mapping,
Trans. Amer. Math. Soc.
{\bf 308} (1988), no.~2, 559--581.
%
\bibitem{Ismail-sieved-III-TAMS-1986}
M. E. H. Ismail,
On sieved orthogonal polynomials. III.
Orthogonality on several intervals,
Trans. Amer. Math. Soc.
{\bf 294} (1986), no.~1, 89--111.
%
\bibitem{Kovalishina-1983}
I. V. Kovalishina,
Analytic theory of a class of interpolation problems,
Izv. Akad. Nauk SSSR Ser. Mat.
{\bf 47} (1983), no.~3, 455--497.
%
%
\bibitem{Marcellan-OP-quadratic-1999}
F. Marcell\'{a}n\ and\ J. Petronilho,
Orthogonal polynomials and quadratic transformations,
Portugal. Math.
{\bf 56} (1999), no.~1, 81--113.
%
\bibitem{Narcowich-R-operators-IndianaJ-1977}
F. J. Narcowich,
$R$-operators. II. On the approximation of certain
operator-valued analytic functions and the Hermitian
moment problem,
Indiana Univ. Math. J.
{\bf 26} (1977), no.~3, 483--513.
%
\bibitem{Simon-moment-self-adjoint-Advances-1998}
B. Simon,
The classical moment problem as a self-adjoint
finite difference operator,
Adv. Math. {\bf 137} (1998), no.~1, 82--203.
%
\bibitem{Simon-CD-kernel-Survey}
B. Simon,
The Christoffel-Darboux kernel,
in {\it Perspectives in partial differential equations,
harmonic analysis and applications},
295--335, Proc. Sympos. Pure Math.,
79, Amer. Math. Soc., Providence, RI.
%
\bibitem{Wall-book}
H. S. Wall,
{\it Analytic Theory of Continued Fractions},
D. Van Nostrand Company,
Inc., New York, NY, 1948.
\end{thebibliography}
\end{document}